\documentclass[A4paper,reqno,oneside,12pt]{amsart}
\usepackage{fullpage}
\usepackage{setspace}
\usepackage{datetime}
\usepackage{tikz}
\usepackage{mathtools}
\usetikzlibrary{er}
\usetikzlibrary{arrows.meta, shapes,snakes,backgrounds,arrows}

\usetikzlibrary{backgrounds}
\tikzset{reverseclip/.style={insert path={(current bounding box.south west)rectangle 
(current bounding box.north east)} }} 

\usepackage{amsmath,amsfonts,amssymb}
\usepackage{verbatim,enumitem,graphics}
\usepackage{microtype}
\usepackage{xcolor}
\usepackage[backref=page]{hyperref}
\hypersetup{%
    colorlinks,
    linkcolor={red!50!black},
    citecolor={green!50!black},
    urlcolor={blue!50!black}
}
\usepackage{dsfont}
\usepackage[abbrev,msc-links,backrefs]{amsrefs} 
\usepackage{doi}
\usepackage{lineno}
\usepackage[normalem]{ulem}
\usepackage[capitalize,nameinlink]{cleveref}
\usepackage[justification=centering]{caption}
\usepackage{adjustbox}
\usepackage{subcaption}

\newcommand{\boundellipse}[3]
{(#1) ellipse (#2 and #3)
}

\newtheorem{theorem}{Theorem}[section]
\newtheorem{lemma}[theorem]{Lemma}
\newtheorem{proposition}[theorem]{Proposition}

\newtheorem{claim}[theorem]{Claim}

\newcommand{\NN}{\mathbb N}

\newcommand{\PP}{\mathbb P}

\newcommand{\EE}{\mathbb E}

\newcommand{\cH}{\mathcal{H}}

\newcommand{\cL}{\mathcal{L}}
\newcommand{\cM}{\mathcal{M}}

\let\epsilon=\varepsilon
\let\phi=\varphi

\begin{document}

\setstretch{1.27}

\title{Note on set representation of Bounded Degree Hypergaphs}
\author[]{Ayush Basu}
\author[]{Griffin Johnston}
\author[]{Vojt\v ech R\"odl}
\author[]{Marcelo Sales}

\address{Department of Mathematics, Emory University, Atlanta, GA, USA}
\email{\{ayush.basu|john.johnston|vrodl\}@emory.edu}
\address{Department of Mathematics, University of California, Irvine, CA, USA}
\email{mtsales@uci.edu}
\thanks{The first and third authors were supported by  by NSF grant DMS 2300347, the third and fourth authors were supported by NSF grant DMS 1764385, and the fourth author was also supported by US Air Force grant FA9550-23-1-0298.}

\begin{abstract}
    In their classical paper, Erd\H{o}s, Goodman and P\'{o}sa \cite{EGP} studied the representation of a graph with vertex set $[n]$ by a family of subsets $S_1,\dots, S_n$ with the property that $\{i,j\}$ is an edge if and only if $S_i\cap S_j\neq \emptyset$. In this note, we consider a similar representation of bounded degree $r$-uniform hypergraphs and establish some bounds for a corresponding problem.
\end{abstract}

\maketitle
\section{Introduction}
A set $S$ \textit{represents} an $r$-uniform hypergraph $G$ if there is a family  $(S_v)_{v\in V(G)}$ of subsets of $S$ such that for any $\{v_1,\dots,v_r\}\subseteq V(G)$,
\begin{align*}
    \{v_1,\dots, v_r\} \in E(G) \iff \left|\bigcap_{i=1}^r S_{v_i}\right| \geq 1.
\end{align*}
One can observe that any $r$-uniform hypergraph can be represented by a finite set and similar to \cite{EGP}, we define \textit{representation number} of an $r$-uniform hypergraph $G$ denoted by $\theta(G)$ as the cardinality of the smallest set $S$ that \textit{represents} $G$. 
\par 
The study of representing graphs (the case where $r=2$) 
can be traced back to the work of
Szpilrajn-Marczewski in \cite{SzpilrajnMarszewski1945}.  
In \cite{EGP}, Erd\H{o}s, Goodman, and P\'osa
introduced the parameter $\theta(G)$ for 2-graphs, 
and proved that $\theta(G) \leq \lfloor n^2/4 \rfloor$ for any graph $G$ on $n$ vertices.  
\par 
For graphs $G$ on $n$ vertices whose complement $\overline{G}$ has bounded maximum degree, i.e., $\Delta(\overline{G})\leq \Delta$,  
Alon \cite{alon1986covering} proved that $\theta(G)\leq c_1\Delta^2\log n$. On the other hand, in \cite{eaton1996graphs} it was shown that for every $\Delta \geq 1$ there are graphs $G$ on $n$ vertices with $\Delta(\overline{G})\leq \Delta$ such that $\theta(G)\geq c_2\frac{\Delta^2}{\log \Delta}\log n$, showing that the upper bound is sharp up to a factor of $\log \Delta$. 
In \cite{rodl2022results}, these results were extended to $r$-uniform hypergraphs. 
\par
The related concept of $k$-representation where $k$ is a positive integer has been studied by a number of authors (for example see \cite{EatonGrable, EatonGouldRodl, chung1994p,furedi97}). For any integer $k>0$, a set $S$ \textit{$k$-represents} an $r$-uniform graph $G$ if there is a family $(S_v)_{v\in V(G)}$ of subsets of $S$ such that for any $\{v_1,\dots,v_r\}\subseteq V(G)$,
\begin{align*}
    \{v_1,\dots, v_r\} \in E(G) \iff \left|\bigcap_{i=1}^r S_{v_i}\right| \geq k.
\end{align*}
The \textit{$k$-representation number} of an $r$-uniform hypergraph $G$, denoted by $\theta_k(G)$, is the cardinality of the smallest set $S$ that \textit{$k$-represents} $G$. Note that for $k=1$, $\theta_1(G) = \theta(G)$ holds. 
\par
It may be natural to ask the following question: given a graph $G$, what is the smallest cardinality of a set $S$ for which there exists a positive integer $k$ such that $S$ $k$-represents $G$? In \cite{eaton1996graphs} the authors studied this question by defining the parameter,
\begin{align*}
    \Tilde{\theta}(G):= \min_{k\in \NN}\theta_k(G),
\end{align*}
for $2$-graphs. In particular, they proved that  $ \Tilde{\theta}(G)\leq c_3 \Delta^2 \log n$ for any graph $G$ on $n$ vertices with $\Delta(G)\leq \Delta$ and that, on the other hand, there exist graphs on $n$ vertices with $\Delta(G)\leq \Delta$ and $ \Tilde{\theta}(G)\geq c_4 \Delta\log (\frac{n}{2\Delta})$. Here we consider the parameter $\Tilde{\theta}(G)$ where $G$ is a bounded degree $r$-uniform hypergraph. 
\par
For a vertex $v$ in $V(G)$ in an $r$-uniform graph $G$, let the degree of $v$, denoted by $d(v)$, be the number of edges that contain $v$, and further let $\Delta(G)$ be the maximum degree of $G$. An $r$-uniform hypergraph $G$ is linear if the intersection of any two edges has size at most 1. 
We will prove the following theorems.
\begin{theorem}[Upper Bound]
\label{MainThm}
    For every $r\geq 3$, there exists a constant $C_r > 0$ and integers $\Delta_0=\Delta_0(r), n_0=n_0(r)$ such that if $G$ is an $r$-uniform hypergraph on $n\geq n_0$ vertices with $\Delta(G)= \Delta \geq \Delta_0$, then 
    \begin{align}
         \Tilde{\theta}(G) \leq C_r \Delta^3 \log n.
    \end{align}
    Further, if $G$ is linear, then 
    \begin{align}
         \Tilde{\theta}(G)\leq C_r \Delta^{2+\frac{1}{r-1}}\log n.
    \end{align}
\end{theorem} 

\begin{theorem}[Lower Bound]
\label{MainThm2}
    For every $r \geq 3$, there exists an integer $n_0=n_0(r)$ such 
    that for every $n\geq n_0$ and $\Delta$, there exists an $r$-uniform hypergraph $G$ on $n$ vertices with $\Delta(G)\leq \Delta$ such that, 
    \begin{align}
         \Tilde{\theta}(G) \geq \frac{\Delta}{4} \log n.
    \end{align}
\end{theorem}

\section{Proof of Upper Bound}
To prove \autoref{MainThm}, we will first decompose the edges of the $r$-uniform hypergraph $G$ into matchings $M_1,\dots, M_L$ for some integer $L$ using \cref{matchingdecomposition}.  We will represent $G$ by a union of $L$ disjoint subsets $S_1,\dots, S_L$ and a family $(R_e)_{e\in E(G)}$ such that $R_e$ is a subset of $S_i$ whenever $e$ is in $M_i$. \cref{Chernoff Lemma} asserts the existence of such families. We will then assign to each $v$ in $V(G)$, the set $S_v$ which will be the disjoint union of all $R_e$ such that $v\in e$ and show that this forms a $k$-representation for some $k\in \NN$. This is done in \cref{MainLemma}.
\\
Given an $r$-uniform graph $G$, let $\chi'(G)$ denote the chromatic index of $G$, defined as the smallest integer $L$ such that $E(G)$ can be decomposed into $L$ matchings. 

\begin{lemma}[Matching Decomposition]
\label{matchingdecomposition}
    If $G$ is an $r$-uniform hypergraph on $n$ vertices and $\Delta(G)\leq \Delta$, then $\chi'(G)\leq L=\Delta\cdot r.$
\end{lemma}

\begin{proof}
    Let $\cL$ be the 2-graph such that $V(\cL) = E(G)$ and, 
    \begin{align*}
        E(\cL) = \{\{e,f\}\subseteq E(G): e\neq f \text{ and } e\cap f \neq \emptyset\}.
    \end{align*}
    Then for any $e$ in $E(G)$, there are at most $(\Delta - 1)\cdot r$ edges $f$ such that $f\neq e$ and $f\cap e\neq \emptyset$. Thus the maximum degree of $\cL$ is $(\Delta -1) r$ and so $\chi(\cL)\leq (\Delta-1) r + 1\leq \Delta r$. For a proper coloring of $\cL$, with $L=\Delta r$ colors, each color class is an independent set in $\cL$ and thus a matching in $G$. Thus $E(G)$ can be decomposed into matchings $M_1,\dots, M_L$, each corresponding to a color class.
\end{proof}
In the following lemma, we will use $x=(a\pm b)$ to denote the inequality, $a-b\leq x\leq a+b$.
\begin{lemma}\label{Chernoff Lemma}
    Let $m, \epsilon, p$ such that $2\leq m\leq r$, $0<\epsilon<1$, $0\leq p\leq 1$. There exists an integer $n_0 = n_0(r)$ such that if $n$ and $t$ are integers satisfying $n\geq n_0$, and 
    \begin{align*}
        t\geq \frac{3(m+1)\log n}{\epsilon^2 p^m},
    \end{align*}
    then there exists a family of subsets $(R_i)_{i\in [n]}$ of a set $S$ of size $t$, such that
    \begin{align}\label{Size of intersections from chernoff}
        \left|\bigcap_{j\in I}R_j\right| = (1\pm \epsilon)p^l t \text{ for every  } I\in [n]^{(l)},
    \end{align}
    whenever $1\leq l\leq m$.
\end{lemma}

\begin{proof}
    Let $n_0= n_0(r)$ be an integer. Wherever necessary, we will assume $n_0$ is large enough. Let  $n,p,\epsilon,t$ be as given above. Let $S$ be a set of size $t$ and  $R_i$ for $i\in [n]$ be random subsets of $S$ with elements chosen independently, each with probability $p$. Fix $1\leq l\leq m$ and let $J\subseteq [n]^{(l)}$, then
    \begin{align*}
        \EE\left[\left|\bigcap_{j\in J} R_j\right|\right] = p^l t.
    \end{align*}
   Since the above random variable has a binomial distribution, we have:
   \begin{align*}
       \PP\left(\left|\bigcap_{j\in J} R_j\right|\neq (1\pm \epsilon)p^l t\right) &< 2\exp\left(-\frac{\epsilon^2p^l t}{3}\right)\\
       & \leq  2\exp\left(-\frac{m+1}{p^{m-l}}\log n\right)\\
       &< 2 n^{-(m+1)}.
   \end{align*}
   Thus, the probability that
   \begin{align*}\label{Size of intersections from chernoff}
        \left|\bigcap_{j\in I}R_j\right| = (1\pm \epsilon)p^l t \text{ for every  } J\subseteq [n]^{(l)},
    \end{align*}
    whenever $1\leq l\leq m$, is at least
    \begin{align*}
        1- \sum_{l=1}^m {n \choose l} 2n^{-(m+1)} > 0,
    \end{align*}
    for $n\geq n_0$ provided $n_0$ is large enough.
\end{proof}

\begin{lemma}\label{MainLemma}
    There exists a constant $A>0$ such that for every integer $r\geq 3$, there are positive integers $n_0=n_0(r)$, and $L_0=L_0(r)$, such that for every $n\geq n_0$ and $L\geq L_0$, if $G$ is an $r$-uniform graph on $n$ vertices with $\chi'(G)\leq L$,
    \[
        \Tilde{\theta}(G) \leq AL^3\log n .
    \]
    Moreover, if $G$ is linear, then 
    \[
        \Tilde{\theta}(G) \leq A(r+1)L^{2 + \frac{1}{r-1}} \log n. 
    \]
\end{lemma}

\begin{proof}
    Fix $r\geq 3$. Let $n_0(r), L_0(r)$ be integers that are assumed to be large enough wherever necessary. Let $G$ be any $r$-uniform hypergraph on $n$ vertices with $\chi'(G)\leq L$. Let $E(G)$ be decomposed into matchings $M_1,\dots, M_L$ with $L\geq L_0$, that is 
    \[
        E(G) = M_1 \sqcup \cdots \sqcup M_L.
    \]
    In what follows, we will give two separate upper bounds for general $r$-uniform hypergraphs and linear $r$-uniform hypergraphs. In each of these cases, we will fix parameters $m,p$ and consider pairwise disjoint subsets $\{S_i:i\in [L]\}$, each of size $t=12(m+1)p^{-m}\log n$, along with families of subsets $(R_e)_{e\in M_i}$ satisfying \cref{Size of intersections from chernoff}. The parameter $m$ will allow us to control the size of $m$-wise intersections for the families $(R_e)_{e\in M_i}$. When $G$ is any $r$-uniform hypergraph (not necessarily linear), we will choose $m=2,p=\frac{1}{4L}$, while when $G$ is a linear $r$-uniform hypergraph, we choose $m=r, p = (\frac{1}{4L})^{\frac{1}{r-1}}$. However, since the analysis for the two cases follow the same steps, we will prove \cref{bounds on intersections} for a general parameter $2\leq m\leq r$. We will then use it to prove the bounds, considering the cases when $G$ is a general $r$-uniform hypergraph (not necessarily linear) and when $G$ is linear. 
    \\[0.3cm]
    \textbf{Construction of Representation:} Fix an integer $m$ such that $2\leq m\leq r$. Let $G$ be an $r$-uniform hypergraph with the matching decomposition $E(G)= M_1\sqcup \cdots \sqcup M_L$. Let $\{S_i\}_{i=1}^L$ be a collection of pairwise disjoint sets of size $t$ and for each $i \in [L]$, let $(R_e)_{e \in M_i}$ be a family of subsets of $S_i$ satisfying \cref{Size of intersections from chernoff}. 
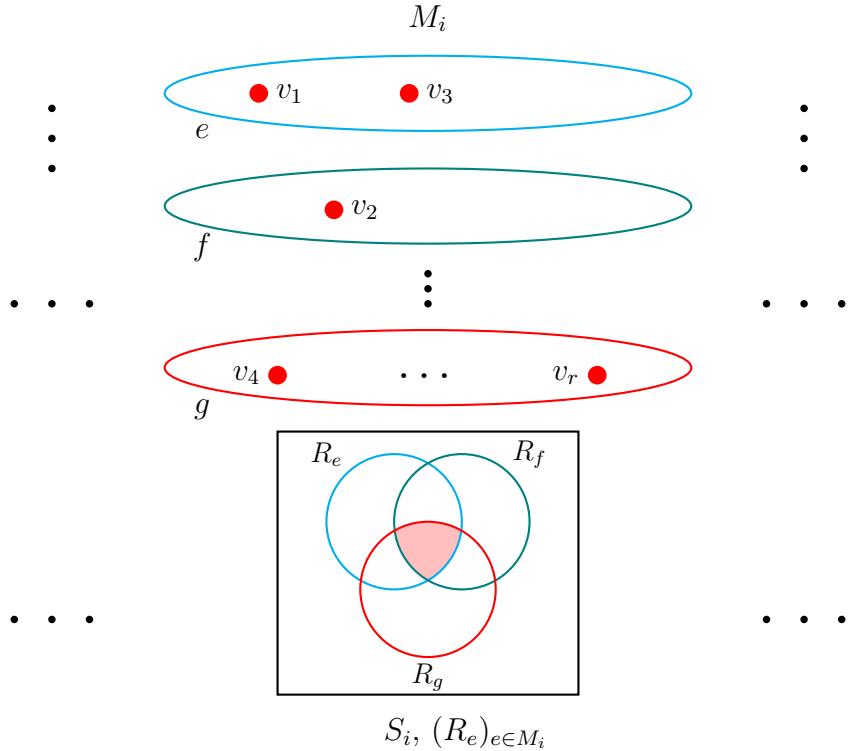
\begin{figure}[b!]
    \centering
      \begin{tikzpicture}[thick,
  fsnode/.style={circle, fill=black, inner sep = 1pt},
   ssnode/.style={circle, fill=red, inner sep = 2.5pt},
    shorten >= -5pt,shorten <= -5pt
]
\begin{scope}[xshift = 0cm]
\node at (0,4) {$M_i$};
\draw[cyan] \boundellipse{0,3}{3.5}{.5};
\draw[teal] \boundellipse{0,1.5}{3.5}{.5};
\draw[red] \boundellipse{0,-.65}{3.5}{.5};
\node[fsnode] at (0,.6) (a) {};
\node[fsnode] at (0,.4) (a) {}; 
\node[fsnode] at (0,.2) (a) {};

\node[ssnode, label={[xshift=1em, yshift= -1em] $v_1$ }] at (-2.25,3) (v1) {};
\node[ssnode, label={[xshift=1em, yshift= -1em] $v_2$ }] at (-1.25,1.45) (v2) {};
\node[ssnode, label={[xshift=1em, yshift= -1em] $v_3$}] at (-.25,3) (v2) {};
\node[ssnode, label={[xshift=-1em, yshift= -1em] $v_r$}] at (2.25,-.75) (v2) {};
\node at (-3,2.5) {$e$};
\node at (-3,.95) {$f$};
\node at (-3,-1.2) {$g$};
\node[ssnode, label={[xshift=-1em, yshift= -1em] $v_4$}] at (-2,-.75) (v2) {};
\node[label = {[xshift = 0em, yshift = -1.1em] \Large$\cdots$} ] at (0,-.75) (dots) {};
\draw[draw=black] (-2,-1.5) rectangle (2,-5 );
\node at (0.5,-5.5) {$S_i$, $(R_e)_{e\in M_i}$};
\end{scope}

\begin{scope}[scale = 0.9, yshift = -3cm, xshift = -0.5cm, every node/.append style={transform shape} ]
\draw[thick, cyan] (0,0) circle (1.0cm);
\draw[thick, teal] (1,0) circle (1.0cm);
\draw[thick, red] (0.5,-1) circle (1.0cm);

\node at (-1,1) {$R_e$};
\node at (2,1) {$R_f$};
\node at (0.5,-2.3) {$R_g$};
\begin{scope}[on background layer]
\clip (0,0) circle (1.0cm);
\clip (1,0) circle (1.0cm);
\fill[pink] (0.5,-1) circle (1.0cm);
\end{scope}

\end{scope}

\begin{scope}[xshift = -5cm, yshift = 0.2cm]
    \node[fsnode] at (-0.5,0) (a) {};
\node[fsnode] at (0,0) (a) {}; 
\node[fsnode] at (0.5,0) (a) {};
\end{scope}

\begin{scope}[xshift = 5cm, yshift = 0.2cm]
    \node[fsnode] at (-0.5,0) (a) {};
\node[fsnode] at (0,0) (a) {}; 
\node[fsnode] at (0.5,0) (a) {};
\end{scope}

\begin{scope}[xshift = 5cm, yshift = -4cm]
    \node[fsnode] at (-0.5,0) (a) {};
\node[fsnode] at (0,0) (a) {}; 
\node[fsnode] at (0.5,0) (a) {};
\end{scope}

\begin{scope}[xshift = -5cm, yshift = -4cm]
    \node[fsnode] at (-0.5,0) (a) {};
\node[fsnode] at (0,0) (a) {}; 
\node[fsnode] at (0.5,0) (a) {};
\end{scope}

\begin{scope}[xshift = 5cm, yshift = 2cm]
\node[fsnode] at (0,0.8) (a) {};
\node[fsnode] at (0,0.4) (a) {}; 
\node[fsnode] at (0,0.0) (a) {};
\end{scope}

\begin{scope}[xshift = -5cm, yshift = 2cm]
\node[fsnode] at (0,0.8) (a) {};
\node[fsnode] at (0,0.4) (a) {}; 
\node[fsnode] at (0,0.0) (a) {};
\end{scope}

\end{tikzpicture}
    \caption{Pairwise disjoint sets $S_i$ with the families $(R_e)_{e\in M_i}$. The shaded area corresponds to $R_e\cap R_f\cap R_g = R(v_1,i)\cap R(v_2,i) \cap \cdots \cap R(v_r,i).$}
    \label{fig:matching_with_sampling}
\end{figure} 
    For any $v \in V(G)$ and $i \in [L]$, let
    \begin{align}
    \label{definitionR(vi)}
        R(v,i) = 
        \begin{cases} 
        R_e & \text{if there exists an $e \in M_i$ such that 
        $v \in e$,} \\
        \emptyset & \text{otherwise}.
        \end{cases}
    \end{align}
    We construct the representation of 
    $G$ as follows. For every $v \in V(G)$, define 
    \begin{align}
    \label{S_v}
        S_v := \bigcup_{i=1}^L R(v,i).
    \end{align}
    Observe that, for any $\{v_1,\dots, v_r\} \subseteq V(G)$, 
    \begin{align}\label{tuple intersection size}
        \left| \bigcap_{j=1}^r S_{v_j} \right| = 
        \sum_{i=1}^L \left| \bigcap_{j=1}^r R(v_j, i) \right|.
    \end{align}
     If there exists a $k \in \mathbb{N}$ such that  
    \begin{align}
    \label{krepresentation}
         \sum_{i=1}^L \left| \bigcap_{j=1}^r R(v_j, i) \right| \geq k \iff \{v_1,\dots, v_r\} \in E(G),
    \end{align}
    then $S_1\sqcup \cdots \sqcup S_L$ $k$-represents $G$ and $\Tilde{\theta}(G)\leq \theta_k(G)\leq |S_1|+\cdots |S_L|= Lt$. We will now find such a $k$ by giving a  lower and upper bound on $|R(v_1,i)\cap \cdots \cap R(v_r,i)|$ for each $i\in [L]$ when $\{v_1,\dots,v_r\}$ is an edge and non-edge respectively.   
    \\[0.3cm]
    \textbf{Bounding the size of intersections $|R(v_1,i)\cap \cdots \cap R(v_r,i)|$:} Note that given a fixed $i\in [L]$, the sets $R(v_j,i)$ are not necessarily distinct for distinct $j$. For example, \cref{fig:matching_with_sampling} depicts the situation when $R(v_1,i)=R(v_3,i)=R_e$. Further, for a fixed $i\in [L]$, since the families $(R_e)_{e\in M_i}$ satisfy \cref{Size of intersections from chernoff}, $|R(v_1,i)\cap \cdots \cap R(v_r,i)|$  ``shrinks" with the number of distinct $R(v_j,i)$ for $j\in \{1,\dots, r\}$. 
    \\
    In particular, if $e=\{v_1,\dots, v_r\}$ is an edge, then it is in some matching $M_i$, and the sets $R(v_j,i)=R_e$ and the size of the intersection, $|R(v_1,i)\cap \cdots \cap R(v_r,i)|$, is roughly $pt$. On the other hand, if $\{v_1,\dots, v_r\}$ is not an edge, then for every matching $M_i$, there are at least two distinct $R(v_j,i)$ for $j\in \{1,\dots, r\}$ and $|R(v_1,i)\cap \cdots \cap R(v_r,i)|$ is at most $p^2t$. \cref{bounds on intersections} below states a slightly stronger version of this observation. Before we state it, it will be convenient to introduce some notation. 
    \\
    Let $\{v_1,\dots,v_r\}$ be an $r$-tuple. Given a matching $M_i$, let
    \begin{align}
        a_i = a_i(\{v_1,\dots, v_r\}) := \Big|\{ e \in M_i \ : \ e \cap \{v_1,\dots, v_r\} \neq \emptyset \}\Big|,
    \end{align}
    i.e, $a_i$ is the number of edges in the matching $M_i$ that intersect $\{v_1,\dots, v_r\}$. Further, let,
    \begin{itemize}
        \item[] $I_1 = I_1(\{v_1,\dots,v_r\}) = \{i\in [L]: \{v_1,\dots, v_r\}\nsubseteq\cup_{e\in M_i}e\}$ and, 
        \item[] $I_2 =  I_2(\{v_1,\dots,v_r\}) = \{i\in [L]: \{v_1,\dots, v_r\}\subseteq\cup_{e\in M_i}e \}$,
    \end{itemize}
    i.e., $I_1$ and $I_2$ are the sets of those $i\in [L]$ such that the union of the edges in $M_i$ do not and do cover the sets $\{v_1,\dots, v_r\}$, respectively. 
    \begin{proposition}
    \label{bounds on intersections}
    For every $i\in [L]$, let $(R_e)_{e\in M_i}$ be a family that satisfies \cref{Size of intersections from chernoff} with a fixed integer $m$ such that $2\leq m\leq r$ and for every $v\in V(G)$, let $R(v,i)$ be as given in \cref{definitionR(vi)}. Then, for every $\{v_1,\dots, v_r\}\in E(G)$, there exists $i\in [L]$ such that
        \begin{align}\label{edge inequality general}
        \left|\bigcap_{j=1}^r R(v_j,i)\right| \geq (1-\epsilon) pt,
    \end{align}
    and for every $\{v_1,\dots,v_r\}\notin E(G)$, 
    \begin{itemize}
        \item If $i\in I_1$, then 
         \begin{align}
    \label{emptysetbound}
        \left| \bigcap_{j=1}^r R(v_j, i) \right| = 0. 
    \end{align}
        \item If $i\in I_2$, and $a_i\leq m$, then 
        \begin{align}
        \label{nonedgeintersection}
        \left|\bigcap_{j=1}^r R(v_j, i)\right| \leq 
        (1+\epsilon)p^{a_i}t. 
        \end{align}
        Further, $a_i \geq 2$.
    \end{itemize}
    \end{proposition}
\begin{proof}
\begin{figure}[t!]
\centering
\begin{subfigure}[b]{.45\textwidth}   
      \begin{tikzpicture}[thick,
  fsnode/.style={circle, fill=black, inner sep = 1pt},
  ssnode/.style={circle, fill=red, inner sep = 2.5pt},
    shorten >= -5pt,shorten <= -5pt, scale = .75
]
\node at (0,4) {$M_i$};
\draw \boundellipse{0,3}{3.5}{.5};
\draw \boundellipse{0,1.5}{3.5}{.5};
\draw \boundellipse{0,0}{3.5}{.5};
\node[fsnode] at (0,-.75) (a) {};
\node[fsnode] at (0,-1) (a) {}; 
\node[fsnode] at (0,-1.25) (a) {};
\node[ssnode, label={[xshift=1em, yshift= -1em]  $v_{j_i}$}] at (-2,-2) (v1) {};
\node[ssnode, label={[xshift=-1.5em, yshift= -1em] $v_{1}$ }] at (-1.25,1.45) (v2) {};
\node[label = {[xshift = 0em, yshift = -1.1em] \Large$\cdots$} ] at (.1,1.45) (dots) {};
\node[ssnode, label={[xshift=1.4em, yshift= -1em] $v_{j_i-1}$ }] at (1.25,1.45) (vr) {};
\node[ssnode, label={[xshift=-1.5em, yshift= -1em] $v_{j_i+1}$ }] at (-1.25,0) (v2) {};
\node[label = {[xshift = 0em, yshift = -1.1em] \Large$\cdots$} ] at (.1,0) (dots) {};
\node[ssnode, label={[xshift=1.4em, yshift= -1em] $v_{r}$ }] at (1.25,0) (vr) {};
\end{tikzpicture}
\centering
    \caption{Case I: $i\in I_1$}
    \label{fig: Case 1}
\end{subfigure} 
\hfill
\begin{subfigure}[b]{.45\textwidth}
    \centering
      \begin{tikzpicture}[thick,
  fsnode/.style={circle, fill=black, inner sep = 1pt},
  ssnode/.style={circle, fill=red, inner sep = 2.5pt},
    shorten >= -5pt,shorten <= -5pt, scale =.75
]
\node at (0,4) {$M_i$};
\draw \boundellipse{0,3}{3.5}{.5};
\draw \boundellipse{0,1.5}{3.5}{.5};
\draw \boundellipse{0,0}{3.5}{.5};
\node[fsnode] at (0,-.75) (a) {};
\node[fsnode] at (0,-1) (a) {}; 
\node[fsnode] at (0,-1.25) (a) {};
\draw \boundellipse{0,-2}{3.5}{.5};
\node[ssnode, label={[xshift=1em, yshift= -1em] $v_1$ }] at (-2.25,3) (v1) {};
\node[ssnode, label={[xshift=1em, yshift= -1em] $v_2$ }] at (-1.25,1.45) (v2) {};
\node[ssnode, label={[xshift=1em, yshift= -1em] $v_3$}] at (-.25,3) (v2) {};
\node[ssnode, label={[xshift=-1em, yshift= -1em] $v_r$}] at (2.25,0) (v2) {};
\node[ssnode, label={[xshift=-1em, yshift= -1em] $v_4$}] at (-2,0) (v2) {};
\node[label = {[xshift = 0em, yshift = -1.1em] \Large$\cdots$} ] at (0,0) (dots) {};
\end{tikzpicture}
    \caption{Case II: $i\in I_2$}
    \label{fig: Case 2}
\end{subfigure} 
\caption{ }
\end{figure}
     If $\{v_1,\dots, v_r\} \in E(G)$, then 
    there is an $i \in [L]$ such that 
    $\{v_1,\dots, v_r\} \in M_i$ and hence for all $1\leq j\leq r$, $R(v_j,i)$ are identical. Consequently, $a_i=1$ and
    \begin{align*}
       \left|\bigcap_{j=1}^r R(v_j,i)\right| \geq (1-\epsilon) pt.
    \end{align*}
Next, we fix $\{v_1,\dots, v_r\} \notin E(G)$, and consider the following cases. Case I (\cref{fig: Case 1}) considers matchings with isolated vertices, and implies \cref{emptysetbound}, while Case II (\cref{fig: Case 2}) considers matchings with no isolated vertices, and implies \cref{nonedgeintersection}.

\begin{description}
    \item[Case I] Let $i\in I_1$, i.e., $\{v_1,\dots, v_r\}\nsubseteq\cup_{e\in M_i}e $. Then there is a $v_{j_i}$ that is not in any edge in $M_i$ (\cref{fig: Case 1}) and thus $R(v_{j_i}, i) = \emptyset$. Consequently,
    \begin{align*}
        \left| \bigcap_{j=1}^r R(v_j, i) \right| = 0. 
    \end{align*}
    \item[Case II] Let $i\in I_2$, i.e., $\{v_1,\dots, v_r\}\subseteq\cup_{e\in M_i}e $. Then every $v_j\in \{v_1,\dots, v_r\}$ is contained in some edge in $M_i$ (\cref{fig: Case 2}). Since $\{v_1,\dots, v_r\}$ is not an edge, there are at least two such edges in $M_i$, and thus $a_i\geq 2$. Further, since $(R_e)_{e\in M_i}$ satisfy \cref{Size of intersections from chernoff}, whenever $l=a_i \leq m$, we have, 
    \[
        \left|\bigcap_{j=1}^r R(v_j, i)\right| \leq 
        (1+\epsilon)p^{a_i}t.  \qedhere
    \]
\end{description}
\end{proof}
Having described our construction of the representation and computed the bounds on $|R(v_1,i)\cap \cdots \cap R(v_r,i)|$ in \cref{bounds on intersections}, we use \cref{Chernoff Lemma} to show that such a construction exists and $k$-\textit{represents} $G$ for some $k\in \NN$. 
\\[0.3cm]
\textbf{General Case: }First consider the case where $G$ is any $r$-uniform hypergraph with $E(G)= M_1\sqcup \cdots \sqcup M_L$. Let $t = \lceil 576L^2\log n \rceil$, $m = 2$, $p = \frac{1}{4L}$, $\epsilon = 1/2$ and $k = \lfloor (1-\epsilon)pt\rfloor$. By \cref{Chernoff Lemma}, there exists pairwise disjoint sets $\{S_i:i\in [L]\}$, each of size $t$, and families of subsets $(R_e)_{e\in M_i}$ of $S_i$, satisfying \cref{Size of intersections from chernoff} with $m=2$. For every $v\in V(G)$, let $S_v$ be as in \cref{S_v}, in our construction of the representation. 
\\
For every $\{v_1,\dots, v_r\}\in E(G)$, \cref{bounds on intersections}, \cref{edge inequality general} implies that,
\begin{align*}
      \left| \bigcap_{j=1}^r S_{v_j} \right| = 
        \sum_{i=1}^L \left| \bigcap_{j=1}^r R(v_j, i) \right|\geq (1-\epsilon)pt\geq k.
\end{align*}
On the other hand, since,
\begin{align*}
L(1+\epsilon)p^2 t = \frac32 \cdot \frac{1}{16L}t < k.
\end{align*}
For every $\{v_1,\dots, v_r\}\notin E(G)$, by \cref{nonedgeintersection}, we have the upper bound,
\begin{align*}
      \left| \bigcap_{j=1}^r S_{v_j} \right| = 
        \sum_{i=1}^L \left| \bigcap_{j=1}^r R(v_j, i) \right|\leq \sum_{i=1}^L(1+\epsilon)p^{a_i} t\leq L(1+\epsilon)p^2 t <k.
\end{align*}
Consequently, $G$ can be $k$-\textit{represented} by the set $S_1\cup S_2\cup \dots \cup S_L$. 
This implies that, for $A= 577$, we have
\[
    \Tilde{\theta}(G)\leq \theta_k(G) \leq Lt \leq AL^3\log n.
\]
\\[0.3cm]
\textbf{Linear Case:} Let $G$ be a linear $r$-uniform hypergraph with $E(G)=M_1\sqcup\cdots \sqcup M_L$. 
\\
Let $t = \lceil 384(r+1)L^{\frac{r}{r-1}}\log n \rceil$, $m = r$, $p = \left(\frac{1}{4L}\right)^{\frac{1}{r-1}}$, $\epsilon = \frac12$ and $k = \lfloor (1-\epsilon) pt\rfloor$. By \cref{Chernoff Lemma}, there exists pairwise disjoint sets $\{S_i:i\in [L]\}$, each of size $t$, and families of subsets $(R_e)_{e\in M_i}$ of $S_i$, satisfying \cref{Size of intersections from chernoff} with $m=r$. For every $v\in V(G)$, let $S_v$ be as in \cref{S_v}, in our construction of the representation. 
\\
For every $\{v_1,\cdots, v_r\}\in E(G)$, in view of \cref{edge inequality general},
\begin{align*}
      \left| \bigcap_{j=1}^r S_{v_j} \right| = 
        \sum_{i=1}^L \left| \bigcap_{j=1}^r R(v_j, i) \right|\geq (1-\epsilon)pt\geq k.
\end{align*}
Now we consider the case where $\{v_1,\dots, v_r\} \notin E(G)$. Note that, if $i\in I_2 = I_2(\{v_1,\dots,v_r\})$, i.e. the edges of $M_i$ cover $\{v_1,\dots, v_r\}$, then $a_i =r$ if and only if each edge $e$ of $M_i$ satisfies $|e\cap \{v_1,\dots, v_r\}|\leq 1$. By linearity of $G$, there are at most ${r\choose 2}$ edges that share a pair of vertices with $\{v_1,\dots, v_r\}$ and, consequently, at most ${r\choose 2}$ matchings $M_i$ with some edge of $M_i$ intersecting $\{v_1,\dots, v_r\}$ in a set of size at least two. Thus, $a_i =r$ for all but at most ${r\choose 2}$ matchings in $I_2$ and, by \cref{bounds on intersections}, $a_i\geq 2$ for the remaining matchings. 
\\
Consequently, for every  $\{v_1,\dots, v_r\} \notin E(G)$, 
\begin{align}
\label{eqn: linearbound}
 \left| \bigcap_{j=1}^r S_{v_j} \right|=\sum_{i=1}^L\left|\bigcap_{j=1}^r R(v_j,i)\right| 
&< (1+\epsilon)\left( \left(L - \tbinom{r}{2}\right)p^rt + 
\tbinom{r}{2}p^2 t \right).
\end{align}
It remains to show  $\left| \cap_{j=1}^r S_{v_j} \right| <k$. Indeed, for large enough $k$ and large enough $L$, the ratio of $\left| \cap_{j=1}^r S_{v_j} \right|$ for a non-edge to an edge is, 
\begin{align*}
    \left(\frac{1+\epsilon}{1 - \epsilon}\right)\frac{(L - \tbinom{r}{2})p^{r}t + \tbinom{r}{2}p^2t}{pt} 
    < 3\left( L\cdot \frac{1}{4L} + \frac{\binom{r}{2}}{(4L)^{\frac{1}{r-1}}}\right) < \frac56.
\end{align*}
Thus, for every  $\{v_1,\dots, v_r\} \notin E(G)$, in view of \cref{eqn: linearbound},
\begin{align*}
 \left| \bigcap_{j=1}^r S_{v_j} \right|=\sum_{i=1}^L\left|\bigcap_{j=1}^r R(v_j,i)\right| 
&< (1+\epsilon)\left( \left(L - \tbinom{r}{2}\right)p^rt + 
\tbinom{r}{2}p^2 t \right) < \frac{5}{6}(1-\epsilon)pt < k.
\end{align*}
Thus, $G$ can be $k$-\textit{represented} by the set $S_1\cup S_2\cup \dots \cup S_L$ and for $A=577$, we have 
\[
    \Tilde{\theta}(G)\leq \theta_k(G) \leq Lt \leq A(r+1)L^{1+\frac{r}{r-1}}\log n= 
    A(r+1)L^{2 + \frac{1}{r-1}}\log n.
\]
\end{proof}
\begin{proof}[Proof of \cref{MainThm}]
    Given $r\geq 3$ and let $C_r = r^3(r+1) A$, $n_0$, $L_0$ be as in \cref{MainLemma}. Let $\Delta_0 = \lceil L_0/r \rceil $. For a graph $G$ on $n\geq n_0$ vertices with maximum degree $\Delta\geq \Delta_0$, by \cref{matchingdecomposition} $\chi'(G)\leq L = \Delta r$. Then by \cref{MainLemma}, 
    \begin{align*}
        \Tilde{\theta}(G)\leq Ar^3 \Delta^3\log n < C_r \Delta^3 \log n,
    \end{align*}
    and if $G$ is linear,
    \begin{align*}
        \Tilde{\theta}(G) \leq A(r+1) r^{2+\frac{1}{r-1}} \Delta^{2+\frac{1}{r-1}} \log n \leq C_r \Delta^{2+\frac{1}{r-1}}\log n.
    \end{align*}
\end{proof}
\section{Proof of Lower Bound}
The proof of the lower bound extends the approach used in \cite{eaton1996graphs} for the case where $r=2$.  
Fix $r\geq 3$. Whenever necessary, we will assume that $n_0$ is a large enough integer. Assume that $n\geq n_0$. Let $\cH^{(r)}(n,\Delta)$ be the collection of $r$-uniform graphs on the vertex set $[n]$ with bounded degree $\Delta$, and let $\cM^{(r)}(n)$ be the collection of all \textit{almost perfect} matchings of $r$-tuples on $[n]$. Each union of $\Delta$ matchings from $\cM^{(r)}(n)$ is a graph on $[n]$ with maximum degree $\Delta\leq n$, and consequently, 
\begin{align}
\label{eqn: numberofgraphs}
    |\cH^{(r)}(n,\Delta)|\geq {|\cM^{(r)}(n)|\choose \Delta} \geq \left(\frac{|\cM^{(r)}(n)|}{\Delta}\right)^{\Delta} \geq \left(\frac{|\cM^{(r)}(n)|}{n}\right)^{\Delta}.
\end{align}
\begin{claim} 
\label{claim: numberofmatchings}
For $r\geq 3$ and $n\geq n_0$, 
    \begin{align*}
        |\cM^{(r)}(n)|\geq \left(\frac{n}{er}\right)^{n/2}
    \end{align*}
\end{claim}

\begin{proof}
 Let $n = qr + s$ where $0\leq s<r$ and $q = \lfloor n/r \rfloor$. We have that $|\cM^{(r)}(n)|$ is at least
    \begin{align*}
        \frac{1}{q!}{n\choose r}{n-r\choose r}\cdots{r+s\choose r} &= \frac{1}{q!}\frac{n!}{(r!)^{q}s!}\geq \frac{n!}{(n/r)!(r^r)^{n/r}r!}.
    \end{align*}
        We use that $n! \geq \sqrt{2\pi n}(n/e)^{n}$ and consequently, we have, for $n\geq n_0(r)$
    \begin{align*}
        \frac{n!}{(n/r)!(r^r)^{n/r}r!}\geq \left(\frac{n}{er}\right)^{n}\frac{\sqrt{2\pi n}}{r^r}\frac{1}{(n/r)!} \geq  \left(\frac{n}{er}\right)^{n}.\frac{1}{(n/r)!}\geq \left(\frac{n}{er}\right)^{n/2}.
    \end{align*}

\end{proof}

\begin{proof}[Proof of \cref{MainThm2}]
Given any integer $t$, there are at most $(2^{t})^n$ distinct $r$-uniform hypergraphs on the vertex set $[n]$ that can be $k$-represented on the set $[t]$. Consequently, if $t$ is such that $|\cH^{(r)}(n,\Delta)|> 2^{tn}$, then there must exist some $G\in \cH^{(r)}(n,\Delta)$ that cannot be $k$-represented by a set of size $t$ for any $k$, and hence $\Tilde{\theta}(G) > t$. 
\\
In view of \cref{eqn: numberofgraphs} and \cref{claim: numberofmatchings},
\begin{align*}
    \log |\cH^{(r)}(n,\Delta)| &> \Delta\log \left(\frac{|\cM^{(r)}(n)|}{n}\right) \geq \Delta\log \left(\frac{\left(\frac{n}{er}\right)^{n/2}}{n} \right)\\
    & = \Delta\cdot\frac{n}{2}\left(\log \left(\frac{n}{n^{2/n}} \right)-\log(er)
\right)\\
    &\geq \Delta\cdot\frac{n}{4}\log n,
\end{align*}
for large enough $n > n_0(r)$. Consequently, for $t=\frac{1}{4}\Delta\log n $, we have that,
\begin{align*}
    |\cH^{(r)}(n,\Delta)|> 2^{tn}.
\end{align*}

\end{proof} 

\section{Concluding Remarks}
In this note we established upper and lower bounds 
on $\Tilde{\theta}(G)$ that differ by a factor of $O(\Delta^2)$, i.e. 
$\Omega(\Delta \log n) \leq \Tilde{\theta}(G) \leq O(\Delta^3 \log n)$. 
Closing the gap between these bounds is a problem 
of interest.
Further, since the lower bound in \cref{MainThm2} 
is nonconstructive 
it would be interesting to find an explicit construction 
that matches or improves our lower bound. 

\section{Acknowledgments}
We would like to thank the anonymous referees for their valuable comments and suggestions.

\bibliography{griffin}

\end{document}